\documentclass[11pt,reqno]{amsart}
\usepackage{amsmath,amsthm,amscd,amsfonts,amssymb,color}
\usepackage{cite}
\usepackage[mathscr]{eucal}

\usepackage[bookmarksnumbered,colorlinks,plainpages]{hyperref}
\newtheorem{theorem}{Theorem}[section]
\newtheorem{lemma}[theorem]{Lemma}
\newtheorem*{Acknowledgement}{\textnormal{\textbf{Acknowledgement}}}

\theoremstyle{definition}
\newtheorem{definition}[theorem]{Definition}

\newtheorem{Open Prob}[theorem]{Open Problem}
\theoremstyle{remark}

\numberwithin{equation}{section}
\def\DJ{\leavevmode\setbox0=\hbox{D}\kern0pt\rlap{\kern.04em\raise.188\ht0\hbox{-}}D}
\begin{document}

\title[Cantor's intersection theorem in the setting of $\mathcal{F}$-metric spaces]{Cantor's intersection theorem in the setting of $\mathcal{F}$-metric spaces}

\author[S.\ Som, L.K.\ Dey,]
{Sumit Som$^{1}$, Lakshmi Kanta Dey$^{2}$}

\address{{$^{1}$} Sumit Som,
                    Department of Mathematics,
                    National Institute of Technology
                    Durgapur, India.}
                    \email{somkakdwip@gmail.com}
%\address{{$^{2}$} Ashis Bera
                    %Department of Mathematics,
                    %National Institute of Technology
                    %Durgapur, India.}
                    %\email{beraashis.math@gmail.com}
\address{{$^{2}$} Lakshmi Kanta Dey,
                    Department of Mathematics,
                    National Institute of Technology
                    Durgapur, India.}
                    \email{lakshmikdey@yahoo.co.in}


\begin{abstract}
This paper deals with an open problem posed by Jleli and Samet in \cite[\, M.~Jleli and B.~Samet, On a new generalization of metric spaces, J. Fixed Point Theory Appl, 20(3) 2018]{JS1}. In \cite[\, Remark 5.1]{JS1} They asked whether the Cantor's intersection theorem can be extended to $\mathcal{F}$-metric spaces or not. In this manuscript we give an affirmative answer to this open question. We also show that the notions of compactness, totally boundedness in the setting of $\mathcal{F}$-metric spaces are equivalent to that of usual metric spaces.
%In this manuscript, we claim that the newly introduced $\mathcal{F}$-metric space \cite[\, M.~Jleli and B.~Samet, On a new generalization of metric spaces, J. Fixed Point Theory Appl, 20(3) 2018]{JS1} is metrizable. Also, we deduce that the notions of convergence, Cauchy sequence, completeness due to Jleli and Samet  for $\mathcal{F}$-metric spaces are equivalent with that of usual metric spaces. Moreover, we assert that the Banach contraction principle in the context of $\mathcal{F}$-metric spaces is a direct consequence of its standard metric counterpart.
%Also, we present an example of $\mathcal{F}$-metric with $(f,\alpha)\in \mathcal{F}\times [0,\infty)$ such that $f$ is not continuous from right at some point in $[0,\infty)$.
\end{abstract}

\maketitle

\setcounter{page}{1}

\centerline{}

\centerline{}

\section{\bf Introduction}
\baselineskip .55 cm
%Many mathematicians are attracted to work on a topic which is more fundamental and has a lot of applications in many diversified fields. One of the major motivations is to generalize or weaken a certain structure and develop new results compatible to the weaker one. Indeed, there are generalizations which  genuinely develop the subject as a whole and also, there are some which contribute nothing new to the literature.
%It is interesting to note that some of the generalizations are not so important as they are equivalent with some other structures which are easier to handle.
%Likewise, to generalize the notion of distance functions, in 1906, Fr$\acute{e}$chet first introduced the concept of metric spaces and a century later, we have numerous generalizations of the metric structure. Intent readers are referred to \cite{D3,M2,HS} and references therein for some relevant extensions. Unfortunately, some of the generalizations become redundant and turn into metrizable merely adding premise to the subject. As in 2007, Huang and Zhang \cite{HZ} introduced the notion of a cone metric space on a positive cone in a Banach space. Following that, a lot of research articles dealt with the setting and evolved the structure with a number of results. Although in 2011, Khani and Pourmahdian \cite{KP} explicitly constructed a metric on a specified cone metric space and proved that cone metric spaces are metrizable.

Recently, Jleli and Samet \cite{JS1} proposed a new generalization of our usual metric space concept. By means of a certain class of functions, the authors defined the notion of an $\mathcal{F}$-metric space. Firstly, we will recall the definition of such kind of spaces. Consider $\mathcal{F}$ be any class of functions $f:(0,\infty)\rightarrow \mathbb{R}$ which satisfy the following conditions:

($\mathcal{F}_1$) $f$ is non-decreasing, i.e., $ 0<s<t\Rightarrow f(s)\leq f(t)$.

($\mathcal{F}_2$) For every sequence $\{t_n\}_{n\in \mathbb{N}}\subseteq (0,+\infty)$, we have
$$\lim_{n\to +\infty}t_n=0 \Longleftrightarrow \lim_{n\to +\infty}f(t_n)=-\infty.$$

%Employing these functions the authors came up with an extension of the concept of usual metric spaces and constructed 

Now we like to give the definition of $\mathcal{F}$-metric space as follows:

\begin{definition} \cite{JS1} \label{D1}
Let $X$ be a non-empty set and $D:X\times X\rightarrow [0,\infty)$ be a given mapping. Suppose that there exists $(f,\alpha)\in \mathcal{F}\times [0,\infty)$ such that
\begin{enumerate}
\item[(D1)] $(x,y)\in X \times X,~~ D(x,y)=0\Longleftrightarrow x=y$.
\item[(D2)] $D(x,y)=D(y,x),~$ $\forall~ (x,y)\in X \times X$.
\item[(D3)] For every $(x,y)\in X\times X$, for each $N\in \mathbb{N}, N\geq2$, and for every $(u_i)_{i=1}^{N}\subseteq X $ with $(u_1,u_N)=(x,y)$, we have
$$D(x,y)>0 \Longrightarrow f(D(x,y))\leq f\left(\sum_{i=1}^{N-1}D(u_i,u_{i+1})\right)+ \alpha.$$
\end{enumerate}
Then $D$ is said to be an $\mathcal{F}$-metric on $X$, and the pair $(X,D)$ is said to be an $\mathcal{F}$-metric space.
\end{definition}
By means of the manuscript \cite{SAL}, the authors proved that this new generalization of metric space is metrizable under a suitable metric $d:X\times X\rightarrow \mathbb{R}$ defined by 
\begin{equation}
d(x,y)=\mbox{inf}\left\{\sum_{i=1}^{N-1}D(u_i, u_{i+1}): N\in \mathbb{N}, N\geq 2, \{u_i\}_{i=1}^{N}\subseteq X~\mbox{with}~(u_1,u_N)=(x,y)\right\}.\label{e}
\end{equation}
They also showed that the notions of Cauchy sequence, completeness, Banach contraction principle are equivalent with that of usual metric spaces. In this manuscript we give an affermative answer to the open question posed by Jleli and Samet in \cite{JS1}.

%In this article, by means of the undermentioned Theorem \ref{FMS}, we assert that the recently proposed $\mathcal{F}$-metric spaces \cite{JS1} are metrizable.  Hence we also state that the new metric induces same notions of convergence, Cauchy sequence and completeness as that of the usual metric spaces. Moreover, we confirm that the celebrated Banach contraction principle in $\mathcal{F}$-metric context can be derived from its existing standard metric counterpart. 

\section{\bf Main Results }
In this section we will prove the Cantor's intersection theorem in the setting of $\mathcal{F}$-metric spaces. But before proving this theorem we will give a lemma which will be needed for proving the theorem. From now on $D$ will denote the $\mathcal{F}$-metric, $d$ will denote the metric defined by \eqref{e}. $\tau_{\mathcal{F}}$ and $\tau_{d}$ denotes the topologies generated by the metrics $D$ and $d$ respectively. Before stating the lemma we first want to introduce the notion of $\mathcal{F}$-boundedness in the setting of $\mathcal{F}$-metric spaces.

\begin{definition}
Let $(X,D)$ be an $\mathcal{F}$-metric space with $(f,\alpha)\in \mathcal{F}\times [0,\infty).$ Then $A\subseteq X$ is said to be $\mathcal{F}$-bounded if $\exists~ M>0$ such that $D(x,y)\leq M~\forall~x,y\in A.$ 
\end{definition}
 
\begin{lemma}\label{e1}
Let $(X,D)$ be an $\mathcal{F}$-metric space with $(f,\alpha)\in \mathcal{F}\times [0,\infty)$ and let $A\subseteq X$ is $\mathcal{F}$-bounded Then $A$ is bounded w.r.t the metric $d$ and $diam_{d}(A)\leq diam_{D}(A).$
\end{lemma}

\begin{proof}
Let $A\subseteq X$ is $\mathcal{F}$-bounded. Then $\exists~ M>0$ such that $D(x,y)\leq M~\forall~x,y\in A.$ Now by the definition of the metric $d$ \eqref{e} we have $$d(x,y)\leq D(x,y)~\forall~x,y\in X.$$ 
$$\Rightarrow d(x,y)\leq M~\forall~x,y\in A.$$ This shows that $A$ is bounded w.r.t the metric $d.$ Proof of the second part follows similarly, so omitted.
\end{proof}

\begin{theorem} [\textbf{Cantor's Intersection Theorem}]
Let $(X,D)$ be an $\mathcal{F}$-metric space with $(f,\alpha)\in \mathcal{F}\times [0,\infty)$. Then $X$ is $\mathcal{F}$-complete if and only if for every decreasing sequence $\{F_n\}_{n\in \mathbb{N}}$ of non-empty, $\mathcal{F}$-closed subsets of $X$ with $diam_{D}(F_n)\rightarrow 0$ as $n\rightarrow \infty$, $\bigcap_{i=1}^{\infty}F_{i}$ contains only one point.
\end{theorem}

\begin{proof}
First of all suppose that $X$ is $\mathcal{F}$-complete. Then $X$ is complete w.r.t the metric $d$ \cite[\, Theorem 2.3 (iii)]{SAL}. Now suppose, $\{F_n\}_{n\in \mathbb{N}}$ is a decreasing sequence of non-empty, $\mathcal{F}$-closed subsets of $X$ with $diam_{D}(F_n)\rightarrow 0$ as $n\rightarrow \infty.$ As $\forall~n\in \mathbb{N},~X\setminus F_n \in \tau_{\mathcal{F}}\Rightarrow X\setminus F_n \in \tau_{d}.$ So $\forall~n\in \mathbb{N}, F_{n}$ is closed w.r.t the metric $d.$ Also by lemma \ref{e1}, $diam_{d}(F_n)\leq diam_{D}(F_n)\rightarrow 0$ as $n\rightarrow \infty.$ So by Cantor's intersection theorem for standard metric spaces we can say that $\bigcap_{i=1}^{\infty}F_{i}$ contains only one point.\\
For the reverse part, suppose $\{F_n\}_{n\in \mathbb{N}}$ is a decreasing sequence of non-empty, $\mathcal{F}$-closed subsets of $X$ with $diam_{D}(F_n)\rightarrow 0$ as $n\rightarrow \infty$ and $\bigcap_{i=1}^{\infty}F_{i}$ contains only one point. So by similar arguments we can say that $\{F_n\}_{n\in \mathbb{N}}$ is a decreasing sequence of non-empty, closed subsets of $X$ w.r.t the metric $d$ with $diam_{d}(F_n)\rightarrow 0$ as $n\rightarrow \infty$ and $\bigcap_{i=1}^{\infty}F_{i}$ contains only one point. So similarly by Cantor's intersection theorem for standard metric spaces we can say that $X$ is complete with respect to the metric $d.$ So $X$ is $\mathcal{F}$-complete by \cite[\, Theorem 2.3 (iii)]{SAL}.
\end{proof}

Now we will prove that the notion of compactness in the setting of $\mathcal{F}$-metric spaces is equivalent with that of usual metric spaces.

\begin{theorem}\label{e2}
Let $(X,D)$ be an $\mathcal{F}$-metric space with $(f,\alpha)\in \mathcal{F}\times [0,\infty)$. Then $A\subseteq X$ is $\mathcal{F}$-compact if and only if $A$ is compact w.r.t the metric $d~ \eqref{e}.$
\end{theorem}

\begin{proof}
First of all suppose that $A\subseteq X$ is $\mathcal{F}$-compact. Let $\{U_{\alpha}\}_{\alpha\in \Lambda}$ be an open cover of $A$ w.r.t the metric $d.$ So $U_{\alpha} \in \tau_{d}~\forall~\alpha \in \Lambda \Rightarrow U_{\alpha} \in \tau_{\mathcal{F}}~\forall~\alpha \in \Lambda.$ As $A\subseteq X$ is $\mathcal{F}$-compact so $\exists$ a finite set $\Lambda_0 \subseteq \Lambda$ such that $A\subseteq \bigcup_{\alpha\in \Lambda_0}U_{\alpha}.$ But as $\tau_{d}=\tau_{\mathcal{F}},$ So $U_{\alpha}\in \tau_{d}~\forall~\alpha \in \Lambda_0.$ This shows that $A$ is compact w.r.t the metric $d.$ For the converse part, the arguments are similar, so omitted.
\end{proof}

\begin{theorem}
Let $(X,D)$ be an $\mathcal{F}$-metric space with $(f,\alpha)\in \mathcal{F}\times [0,\infty)$. Then $X$ is $\mathcal{F}$-compact if and only if for every collection $\{F_{\alpha}\}_{\alpha \in \Lambda}$ of $\mathcal{F}$-closed sets, having the finite intersection property, $\bigcap_{\alpha\in \Lambda}F_{\alpha}\neq \phi.$ 
\end{theorem}

\begin{proof}
Suppose that $X$ is $\mathcal{F}$-compact. So by Theorem \ref{e2} we can say that $X$ is compact w.r.t the metric $d.$ Now suppose $\{F_{\alpha}\}_{\alpha \in \Lambda}$ be a collection of $\mathcal{F}$-closed sets having the finite intersection property. So $\{F_{\alpha}\}_{\alpha \in \Lambda}$ will be a collection of closed sets w.r.t the metric $d$ having the finite intersection property. As $X$ is compact so we must have $\bigcap_{\alpha\in \Lambda}F_{\alpha}\neq \phi.$\\

For the reverse part the arguments are similar and follows from Theorem \ref{e2}.
\end{proof}

The authors defined the concept of $\mathcal{F}$-totally boundedness in \cite{JS1} in the setting of $\mathcal{F}$-metric spaces. But our next theorem ensures that the concept of $\mathcal{F}$-totally boundedness is equivalent to that of usual metric spaces.

\begin{theorem}\label{t1}
Let $(X,D)$ be an $\mathcal{F}$-metric space with $(f,\alpha)\in \mathcal{F}\times [0,\infty)$. If $A\subseteq X$ is $\mathcal{F}$-totally bounded if and only if $A$ is totally bounded w.r.t to the metric $d.$ 
\end{theorem}

\begin{proof}
First of all suppose that $A\subseteq X$ is $\mathcal{F}$-totally bounded. Let $\varepsilon>0.$ So $\exists$ a finite set $\{a_1,a_2,\dots, a_n\}\subseteq A$ such that $A\subseteq \bigcup_{i=1}^{n}B_{D}(a_i,\varepsilon).$ Now by the definition of the metric $d$ we have $B_{D}(a_i,\varepsilon)\subseteq B_{d}(a_i,\varepsilon)~\forall~i\in \{1,2,\dots,n\}.$ So we have $A\subseteq \bigcup_{i=1}^{n}B_{d}(a_i,\varepsilon).$ This shows that $A$ is totally bounded w.r.t to the metric $d.$

For the second part, suppose that $A$ is totally bounded w.r.t to the metric $d.$ Let $\varepsilon>0.$ In \cite[\, Theorem 3.1]{JS1}, the authors showed that for any $\varepsilon*>0,x,y\in X,y\neq x$ 
\begin{equation}
f(D(x,y))\leq f(d(x,y)+\varepsilon*)+\alpha. 
\label{e3}
\end{equation}
Now by $\mathcal{F}_2$ condition, for $(f(\varepsilon)-\alpha)$ there exists a $\delta>0$ such that if $0<t<\delta$ then $f(t)<f(\varepsilon)-\alpha.$ As $A\subseteq X$ is totally bounded w.r.t the metric $d$ so for $\frac{\delta}{2}>0~\exists$ a finite set $\{b_1,b_2,\dots, b_t\}\subseteq A$ such that $A\subseteq \bigcup_{i=1}^{t}B_{d}(b_i,\frac{\delta}{2}).$ Now we will show that $B_{d}(b_i,\frac{\delta}{2})\subseteq B_{D}(b_i,\varepsilon)~\forall~i\in \{1,2,\dots,t\}.$ Let $y\in B_{d}(b_i,\frac{\delta}{2})$ and $|B_{d}(b_i,\frac{\delta}{2})|=1.$ Then $y\in B_{D}(b_i,\varepsilon).$ On the other hand let $|B_{d}(b_i,\frac{\delta}{2})|>1$ and $y\in B_{d}(b_i,\frac{\delta}{2}),y\neq b_{i}.$ Then $d(y,b_i)<\frac{\delta}{2}.$ From \eqref{e3} we have $$f(D(y,b_i))\leq f(d(y,b_i)+\frac{\delta}{2})+\alpha.$$
$$\Rightarrow f(D(y,b_i))<f(\varepsilon)$$
$$\Rightarrow D(y,b_i)<\varepsilon.$$
So $B_{d}(b_i,\frac{\delta}{2})\subseteq B_{D}(b_i,\varepsilon)~\forall~i\in \{1,2,\dots,t\}.$ This shows that $A\subseteq \bigcup_{i=1}^{t}B_{D}(b_i,\varepsilon).$ So $A$ is $\mathcal{F}$-totally bounded.\\
\end{proof}

In \cite[\, Proposition 4.9 (ii)]{JS1}, Authors showed that if $A\subseteq X$ is $\mathcal{F}$-compact then $A\subseteq X$ is $\mathcal{F}$-totally bounded but did not say about the converse. In the next theorem we will consider this converse part.

\begin{theorem}
Let $(X,D)$ be an $\mathcal{F}$-metric space with $(f,\alpha)\in \mathcal{F}\times [0,\infty)$. Then following are equivalent:\\
(i) $X$ is $\mathcal{F}$-complete and $\mathcal{F}$-totally bounded.\\
(ii) $X$ is $\mathcal{F}$-compact.\\
(iii) $X$ is compact w.r.t the metric $d.$\\
(iv) $X$ is complete and totally bounded w.r.t the metric $d.$
\end{theorem}

\begin{proof}
(i)$\Rightarrow$(ii) First of all suppose that $X$ is $\mathcal{F}$-complete and $\mathcal{F}$-totally bounded. Then by \cite[\, Theorem 2.3(iii)]{SAL} we can conclude that $X$ is complete w.r.t the metric $d$ and by Theorem \ref{t1} $X$ is totally bounded w.r.t the metric $d.$ This implies that $X$ is compact w.r.t the metric $d$ and by Theorem \ref{e2} we can conclude that $X$ is $\mathcal{F}$-compact.\\

(ii)$\Rightarrow$(i) Now suppose that $X$ is $\mathcal{F}$-compact. So by Theorem \ref{e2},$X$ is compact w.r.t the metric $d$. This implies that $X$ is complete w.r.t the metric $d.$ So by \cite[\, Theorem 2.3(iii)]{SAL} we can conclude that $X$ is $\mathcal{F}$-complete. The other part already proved in \cite{JS1}.

(ii)$\Rightarrow$(iii) follows from the theorem \ref{e2}.

(iii)$\Rightarrow$(iv) follows from the theory of standard metric spaces.

(iv)$\Rightarrow$(i) follows from \cite[\, Theorem 2.3(iii)]{SAL} and theorem \ref{t1}.
\end{proof}

\begin{Acknowledgement}
 The Research is funded by the Council of Scientific and Industrial Research (CSIR), Government of India under the Grant Number: $25(0285)/18/EMR-II$. 
\end{Acknowledgement}

\bibliographystyle{plain}

\begin{thebibliography}{10}

%\bibitem{B1}
%S.~Banach.
%\newblock Sur les op\'erations dans les ensembles abstraits et leur application
 % aux \'equations int\'egrales.
%\newblock {\em Fund. Math.}, 3:133--181, 1922.

%\bibitem{CMDK}
%A.~Chanda, S.~Mondal, L.K. Dey, and S.~Karmakar.
%\newblock ${C}^*$-algebra-valued contractive mappings with its application to
%  integral equations.
%\newblock {\em Indian J. Math.}, 59(1):107--124, 2017.

%\bibitem{C}
%B.S. Choudhury.
%\newblock Unique fixed point theorem for weakly {C}-contractive mappings.
%\newblock {\em Kathmandu Univ. J. Sci. Engg. Tech.}, 5(1):6--13, 2009.

%\bibitem{DD}
%P.~Das and L.K. Dey.
%\newblock Fixed point of contractive mappings in generalized metric spaces.
%\newblock {\em Math. Slovaca}, 59(4):499--504, 2009.



%\bibitem{E1}
%M.~Edelstein.
%\newblock On fixed and periodic points under contractive mappings.
%\newblock {\em J. Lond. Math. Soc.}, 37(1):74--79, 1962.

%\bibitem{GDMR}
%H.~Garai, L.K. Dey, P.~Mondal, and S.~Radenovi'{c}.
%\newblock Nemytzki-{E}delstein-{S}uzuki type results in $b_{v}(s)$-metric
 % spaces.
%\newblock Preprint.

%\bibitem{GRRS}
%R.~George, S.~Radenovi\'{c}, K.P. Reshma, and S.~Shukla.
%\newblock Rectangular $b$-metric space and contraction principles.
%\newblock {\em J. Nonlinear Sci. Appl.}, 8(6):1005--1013, 2015.

%\bibitem{G}
%M.~Geraghty.
%\newblock On contractive mappings.
%\newblock {\em Proc. Amer. Math. Soc.}, 40(2):604--608, 1973.



\bibitem{JS1}
M.~Jleli and B.~Samet.
\newblock On a new generalization of metric spaces.
\newblock {\em J. Fixed Point Theory Appl. https://doi.org/10.1007/s11784-018-0603-9}, 20(3), 2018.

%\bibitem{JS}
%M.~Jleli and B.~Samet.
%\newblock A generalized metric space and related fixed point theorems.
%\newblock {\em Fixed Point Theory Appl.}, 2015:61, 2015.

%\bibitem{R5}
%R.~Kannan.
%\newblock Some results on fixed points- {II}.
%\newblock {\em Amer. Math. Monthly}, 76:405--408, 1969.



%\bibitem{KSS}
%M.S. Khan, M.~Swalech, and S.~Sessa.
%\newblock Fixed point theorems by altering distances between the points.
%\newblock {\em Bull. Austral. Math. Soc.}, 30(1):1--9, 1984.

\bibitem{M2}
J.R. Munkres.
\newblock {\em Topology}.
\newblock Pearson, 1974.

\bibitem{SAL}
S. Som, A. Bera and L. K Dey.
\newblock Some remarks on the Metrizability of $\mathcal{F}$-metric spaces.
\newblock {\em arXiv:1808.02736}.

%\bibitem{NA1}
%A.D. Nezhad and Z.~Aral.
%\newblock The topology of {GB}-metric spaces.
%\newblock {\em ISRN Math. Anal.}, 2011, 2011.
%\newblock Article ID 523453.

%\bibitem{S3}
%T.~Suzuki.
%\newblock A new type of fixed point theorem in metric spaces.
%\newblock {\em Nonlinear Anal. Model. Control}, 71(11):5313--5317, 2009.

%\bibitem{TFAS}
%M.~Tavakoli, A.P. Farajzadeh, T.~Abdeljawad, and S.~Suantai.
%\newblock Some notes on cone metric spaces.
%\newblock {\em Thai J. Math.}, 16(1):229--242, 2018.

\end{thebibliography}

\end{document}